\documentclass{amsart}
\usepackage{amssymb}
\usepackage{mathrsfs}
\usepackage{amsmath}
\usepackage{amsfonts}
\usepackage{pifont}
\usepackage{graphicx,amssymb,mathrsfs,amsmath}
\usepackage{tocvsec2}
\usepackage{color}
\newtheorem{thm}{Theorem}[section]

\newtheorem{lem}[thm]{Lemma}

\theoremstyle{definition}

\theoremstyle{remark}
\newtheorem{rem}[thm]{Remark}
\numberwithin{equation}{section}

\begin{document}
\title[New inequalities for the function $y=t\ln t$]{New inequalities for the function $y=t\ln t$}

\author{Marko Kosti\' c}
\address{Faculty of Technical Sciences,
University of Novi Sad,
Trg D. Obradovi\' ca 6, 21125 Novi Sad, Serbia}
\email{marco.s@verat.net}

\begin{abstract}
The main aim of this note, which can be viewed as a certain addendum to the paper \cite{2020}, is to propose several new inequalities for the function $y=t\ln t.$ We consider the local behaviour of this function near the point $t=1,$ as well as the global behaviour of this function on the intervals $[1,\infty)$ and $(0,1].$  
\\[2mm] {\it AMS Mathematics Subject Classification $(2010)$}: 26D20, 26D07, 33B30.
\\[1mm] {\it Key words and phrases:}  Logarithmic inequalities, differential calculus.
\end{abstract}

\maketitle

\section{Introduction}\label{intro}

The reading of paper \cite{2020} by C. Chesneau and Y. J. Bagul
has strongly infuenced us to write this note.  In Theorem \ref{platani}, we give
new abstract local bounds for the function $y=t\ln t$ near the point $t=1.$ The obtained inequalities can be used to improve the main results of paper \cite{2020}, Proposition 1 and Proposition 2. We also present an interesting result with regards to these propositions, which claims that there is no rational real function which intermediates the functions $\ln (1+x)$ and $f(x)/\sqrt{x+1}$ for $x\geq 0$ ($x\in (-1,0]$); here and hereafter,
$$
f(x):=\pi +\frac{1}{2}(4+\pi)x-2(x+2)\arctan \sqrt{x+1},\quad x\geq -1.
$$

The following inequalities are well known (see also \cite[Problem 3.6.19, p. 274]{mitro} and \cite{nova}):
\begin{align}\label{bodoviiljah}
\ln (1+x) \leq \frac{x}{\sqrt{x+1}},\quad x\geq 0, \ \ \mbox{  }\ \ \ln (1+x)\leq \frac{x(2+x)}{2(1+x)},\quad x\geq 0,
\end{align}
\begin{align}\label{bodoviiljahu}
\ln (1+x)\leq \frac{x(6+x)}{2(3+2x)},\quad x\geq 0\ \ \mbox{ and }\ \ \ln (1+x)\leq \frac{(x+2)\bigl [(x+1)^{3}-1\bigr]}{3(1+x)\bigl[ (x+1)^{2}+1 \bigr]},\quad x\geq 0.
\end{align}
Taken together, the first inequality in \eqref{bodoviiljah} and the second inequality in \eqref{bodoviiljahu} are known in the existing literature as Karamata's inequality \cite{karamata}. 
As clarified in \cite{2020}, all these inequalities are weaker than the inequality:
\begin{align}\label{bodovi}
\ln (1+x) \leq \frac{f(x)}{\sqrt{x+1}},\quad x\geq 0.
\end{align}

This inequality has been proved in \cite[Proposition 1]{2020}. In \cite[Proposition 2]{2020}, the authors have proved that  
\begin{align}\label{bodovii}
\ln (1+x) \geq \frac{f(x)}{\sqrt{x+1}},\quad x\in (-1,0],
\end{align}
as well.

Our approach leans heavily on the use of substitution $t=\sqrt{x+1}.$ Then the inequalities \eqref{bodovi} and \eqref{bodovii} become
\begin{align*}
2\ln t\leq \frac{f\bigl( t^{2}-1 \bigr)}{t},\quad t\geq 1 \ \ \mbox{ and }\ \ 2\ln t\geq \frac{f\bigl( t^{2}-1 \bigr)}{t},\quad t\in (0, 1],
\end{align*}
i.e.,
\begin{align}\label{mb}
2t\ln t\leq f\bigl( t^{2}-1 \bigr),\quad t\geq 1 \ \ \mbox{ and }\ \ 2t\ln t\geq f\bigl( t^{2}-1 \bigr),\quad t\in (0, 1].
\end{align}
We can prove \eqref{mb} in the following way.
Notice that 
$$
\Biggl[\ln t - \Biggl(\frac{1}{2}(4+\pi)t-2t\arctan t-2\Biggr)\Biggr]^{\prime \prime}(t)=-\bigl(t^{2}-1\bigr)^{2}t^{-2}\bigl(t^{2}+1\bigr)^{-2},\quad t>0.
$$
Using an elementary argumentation,
this estimate implies
$$
\ln t \leq \frac{1}{2}(4+\pi)t-2t\arctan t-2,\quad t>0.
$$ 
Define $R(t):=2t\ln t-f(t^{2}-1),$ $t>0.$ Since $R^{\prime}(t)=2(1+\ln t)-(4+\pi)t+4t\arctan t+2,$ $t>0,$ the previous inequality yields $R^{\prime}(t)\leq 0,$ $t>0$ and \eqref{mb}. Moreover, by taking the limit of function $R(\cdot)$ as $t\rightarrow 0+,$ we get that 
$2t\ln t-f(t^{2}-1) \in (2-(\pi/2),0]$ for $t\in (0,1].$

In this paper, we will first generalize the inequalities in \eqref{mb} by considering the local behaviour of the function $y=t\ln t$ near the point $t=1.$ 
We will use the following simple lemmae, which is known from the elementary courses of mathematical analysis:

\begin{lem}\label{daleko}
Suppose $t_{0}\in {\mathbb R},$ $a>0,$ $n\in {\mathbb N}$ and function
$f :(t_{0}-a,t_{0}+a) \rightarrow {\mathbb R}$
is $2n$-times differentiable. If $f^{(i)}(t_{0})=0$ for all $i=1,\cdot \cdot \cdot, 2n-1$ and $f^{(2n)}(t_{0})>0$ ($f^{(2n)}(t_{0})<0$), then the function $y=f(t)$
has a local minimum (maximum) at 
$t=t_{0}.$
\end{lem}

\begin{lem}\label{dalekor}
We have
$$
(\arctan x)^{(n)}=\frac{(-1)^{n-1}(n-1)!}{(1+x^{2})^{n/2}}\sin (n\pi/2-n \arctan x),\quad x\in {\mathbb R},\ n\in {\mathbb N}.
$$
\end{lem}

After that, we will prove the following result with regards to \cite[Proposition 1, Proposition 2]{2020}:

\begin{thm}\label{ihaj}
\begin{itemize}
\item[(i)]
There do not exist real polynomials $P(\cdot)$ and $Q(\cdot)$ such that $Q(x) \neq 0$ for $x\geq 0$ and
\begin{align}\label{mbswq}
\ln (1+x) \leq \frac{P(x)}{Q(x)} \leq \frac{f(x)}{\sqrt{x+1}},\quad x\geq 0.
\end{align}
\item[(ii)] There do not exist real polynomials $P(\cdot)$ and $Q(\cdot)$ such that $Q(x) \neq 0$ for $x\in (-1, 0]$ and
\begin{align}\label{mbswqq}
\ln (1+x) \geq \frac{P(x)}{Q(x)} \geq \frac{f(x)}{\sqrt{x+1}},\quad x\in (-1,0].
\end{align}
\end{itemize}
\end{thm}

\section{The main results and their proofs}\label{main}

We start this section by stating the following result:

\begin{thm}\label{platani}
Suppose that $a\in (0,1),$ $P: (1-a,1+a) \rightarrow {\mathbb R}$ is a function and $P(1)=0.$
Then the following holds:
\begin{itemize}
\item[(i)] If $P^{\prime}(1)\geq 2$ and there exists an odd natural number $n$ such that 
$P(\cdot)$ is $(n+2)$-times differentiable,
$P^{(n+2)}(1)+2(-1)^{n+1}n!>0$ and
$$
P^{(j)}(1)+2(-1)^{j+1}(j-2)!=0\ \ \mbox{ for all }\ \ j=2,3,\cdot \cdot \cdot, n+1,
$$
then there exists a real number $\zeta \in (0,a]$ such that 
\begin{align}\label{dr}
2t\ln t \leq P(t),\quad t\in [1,1+\zeta]\ \ \mbox{ and }\ \ 2t\ln t\geq P(t),\quad t\in [1-\zeta,1].
\end{align}
\item[(ii)] Assume that there exists an 
even natural number $n\geq 6$ such that $P(\cdot)$ is $(n+1)$-times differentiable, $P^{(n+1)}(1)+2(-1)^{n}(n-1)!>0$ and
$$
P^{(j)}(1)+2(-1)^{j+1}(j-2)!=0\ \ \mbox{ for all }\ \ j=1,2,\cdot \cdot \cdot, n.
$$
Then
there exists a real number $\eta \in (0,a]$ such that
\begin{align}
\notag 2t\ln t\leq P(t) &\leq f\bigl(t^{2}-1\bigr),\quad t\in [1,1+\eta]\ \ 
\\\label{drr}& \mbox{ and }\ \ 2t\ln t\geq P(t) \geq f\bigl(t^{2}-1\bigr),\quad t\in [1-\eta,1].
\end{align}
\item[(iii)] Assume that there exists an 
even natural number $n\geq 6$ such that $P(\cdot)$ is $(n+1)$-times differentiable, 
\begin{align}\label{pokazna}
P^{(j)}(1)+2(-1)^{j+1}(j-2)!=0\ \ \mbox{ for all }\ \ j=1,2,3,4,
\end{align}
$$
P^{(n+1)}(1)+4\Biggl[ \frac{(-1)^{n}n!}{2^{(n+1)/2}}\sin ((n+1) \pi/4)+\frac{(-1)^{n+1}n!}{2^{n/2}}\sin (n \pi/4) \Biggr]<0
$$
and, for every $j=5,6,\cdot \cdot \cdot, n,$ 
\begin{align}\label{pokaznaa}
P^{(j)}(1)+4\Biggl[ \frac{(-1)^{j-1}(j-1)!}{2^{j/2}}\sin (j \pi/4)+\frac{(-1)^{j}(j-1)!}{2^{(j-1)/2}}\sin ((j-1) \pi/4) \Biggr]=0.
\end{align}
Then there exists a real number $\eta \in (0,a]$ such that \eqref{drr} holds.
\item[(iv)] If $P(\cdot)$ is five times differentiable, \eqref{pokazna} holds
and $P^{(v)}(1)\in (-12,-8),$
then there exists a real number $\eta \in (0,a]$ such that \eqref{drr} holds.
\end{itemize} 
\end{thm}

\begin{proof}
Define $G(t):=P(t)-2t\ln t,$ $t>0.$ Then, for every real number $t>0,$ we have
$G^{\prime}(t)=P^{\prime}(t)-2(1+\ln t),$ $G^{\prime \prime}(t)=P^{\prime \prime}(t)-(2/t)$
and $G^{(n)}(t)=P^{(n)}(t)+2(-1)^{n+1}(n-2)!\cdot t^{1-n},$ $n\geq 3.$ The assumptions made in (i) imply that   
$G^{\prime}(1)\geq 0,$ $(G^{\prime})^{(j)}(1)=0$ for $1\leq j\leq n$ and $(G^{\prime})^{(n+1)}(1)>0.$ Applying Lemma \ref{daleko}, 
we get that the function $t\mapsto G^{\prime}(t)$ has a local minimum at $t=1.$ Since 
$G^{\prime}(1) \geq 0,$ we get that the function $t\mapsto G^{\prime}(t)$ is non-negative in an open neighborhood of point $t=1,$ so that the mapping $t\mapsto G(t)$ is increasing in an open neighborhood of point $t=1.$ 
This finishes the proof of (i). For the proof of (ii), define $Q(t):=P(t)-f(t^{2}
-1),$ $t>0.$ Then a simple computation yields that, for every real number $t>0,$ we have
$Q^{\prime}(t)=P^{\prime}(t)-(4+\pi)t+4t\arctan t +2$ and $Q^{\prime \prime}(t)=P^{\prime \prime}(t)-(4+\pi)+4\arctan t +\frac{4t}{t^{2}+1}.$ Using Leibniz rule and Lemma \ref{dalekor}, for every real number $t>0$ and for every natural number $n\geq 3,$ we can show that
\begin{align*}
Q^{(n)}(t)&=P^{(n)}(t)+4\bigl[ \cdot \arctan \cdot  \bigr]^{(n-1)}(t)
\\&=P^{(n)}(t)+4\Biggl[ t\frac{(-1)^{n-1}(n-1)!}{(1+t^{2})^{n/2}}\sin (n\pi/2-n \arctan t)
\\&+ \frac{(-1)^{n}(n-1)!}{(1+t^{2})^{(n-1)/2}}\sin ( (n-1)\pi/2-(n-1) \arctan t)\Biggr].
\end{align*}
Arguing as in the proof of (i), we have that
$(Q^{\prime})^{(j)}(1)=0$ for $j=0,1,2,3$ and $(Q^{\prime})^{(4)}(1)<0;$
hence,
the function $t\mapsto Q^{\prime}(t)$ has a local maximum at $t=1$ and the mapping $t\mapsto Q(t)$ is decreasing in an open neighborhood of point $t=1.$ Similarly, 
$(G^{\prime})^{(j)}(1)=0$ for $j=0,1,2,\cdot \cdot \cdot,n-1$ and $(G^{\prime})^{(n)}(1)>0;$
hence,
the function $t\mapsto G^{\prime}(t)$ has a local minimum at $t=1$ and the mapping $t\mapsto G(t)$ is increasing in an open neighborhood of point $t=1.$ This completes the proof of (ii). The proof of (iii) can be deduced similarly, by interchanging the roles of $G(t)$ and $Q(t).$
If the assumptions of (iv) holds, then we can apply Lemma \ref{daleko}, with $n=2,$ in order to see that the function $t\mapsto G^{\prime}(t)$ has a local minimum at $t=1,$ as well as
the function $t\mapsto G^{\prime}(t)$ is non-negative in an open neighborhood of point $t=1;$ hence, the mapping $t\mapsto G(t)$ is increasing in an open neighborhood of point $t=1.$ Similarly, we can show that the mapping $t\mapsto Q(t)$ is decreasing in an open neighborhood of point $t=1.$
The proof of the theorem is thereby complete.
\end{proof}

\begin{rem}\label{izvodinjo}
Define $H(t):=f(t^{2}-1),$ $t\in {\mathbb R}.$ Concerning the conditions used in Theorem \ref{platani}, it is worth noting that
the function $H(\cdot)$ satisfies 
 $H(1)=0,$ $H^{\prime}(1)=H^{\prime \prime}(1)=2,$ $H^{\prime \prime \prime}(1)=-2,$ $H^{(iv)}(1)=4$ and
$H^{(v)}(1)=-8.$ This implies that the values of terms appearing at the right hand sides of \eqref{pokazna} and \eqref{pokaznaa} coincide for $j=1,2,3,4$ and differ for $j=5$ (observe that $G^{(v)}(1)=P^{(v)}(1)+12$).
\end{rem}

\begin{rem}\label{mbb}
The parts (ii)-(iv) of Theorem \ref{platani} ensure the existence of a large class of elementary functions for which we can further refine the inequalities in \eqref{mb} locally around the point $t=1.$  Compared with the function $H(\cdot)$, the most simplest example of function which provides a better estimate describing the local behaviour of function $y=t\ln t$ around the point $t=1$ is given by the function $t\mapsto H(t)-\epsilon (t-1)^{5},$ $t>0$, where $\epsilon \in (0,1/30).$  
\end{rem}

Concerning the global behaviour of function $y=t\ln t,$ $t>0,$ it is clear that the inequalities in \eqref{mb} give some very uninteresting estimates with regards to the asymptotic behaviour of function $y=t\ln t$ when $t\rightarrow +\infty$ or $t\rightarrow 0+;$ on the other hand, the importance of estimate \eqref{mb} lies in 
the fact that it gives some bounds for
the behaviour of function $y=t\ln t$ on any compact interval $[a,b]$, where $0<a<1<b$.
It is clear that there exists a large class of infinitely differentiable functions
$P  : (0,\infty) \rightarrow {\mathbb R}$ such that 
\begin{align}
\notag 2t\ln t\leq P(t) &\leq f\bigl(t^{2}-1\bigr),\quad t\geq 1\ \ 
\\\label{drrr}& \mbox{ and }\ \ 2t\ln t\geq P(t) \geq f\bigl(t^{2}-1\bigr),\quad t\in (0,1].
\end{align}
Finding new elementary functions $P(\cdot)$ for which the equation \eqref{drrr} holds is without scope of this paper.

We close the paper by giving the proof of Theorem \ref{ihaj}:\vspace{0.15cm}

{\bf Proof of Theorem \ref{ihaj}.} Suppose that \eqref{mbswq} holds for some real polynomials $P(\cdot)$ and $Q(\cdot)$ such that $Q(x) \neq 0$ for $x\geq 0.$ 
Without loss of generality, we may assume that $Q(x)>0,$ $x\geq 0.$
Using the substitution $t=\sqrt{x+1}, $ we get that
$$
2\ln t \leq \frac{P(t^{2}-1)}{Q(t^{2}-1)}\leq \frac{f(t^{2}-1)}{t},\quad t\geq 1.
$$
If $P(t)=\sum_{j=0}^{n}a_{j}t^{j}$ and $Q(t)=\sum_{j=0}^{m}b_{j}t^{j}$ for some non-negative integers $m,\ n$ and some real numbers $a_{j}, \ b_{j}$ ($a_{n}b_{m}\neq 0;$ clearly, we cannot have $P(x)\equiv 0$),  we get
\begin{align}\label{akotreba}
t\sum_{j=0}^{n}a_{j}\bigl(t^{2}-1\bigr)^{j} \leq f\bigl(t^{2}-1\bigr)\sum_{j=0}^{m}b_{j}\bigl(t^{2}-1\bigr)^{j},\quad t\geq 1
\end{align}
and
\begin{align}\label{akotrebae}
\sum_{j=0}^{n}a_{j}\bigl(t^{2}-1\bigr)^{j} \geq 2\ln t\sum_{j=0}^{m}b_{j}\bigl(t^{2}-1\bigr)^{j},\quad t\geq 1.
\end{align}
Since $f(t^{2}-1) \sim (2-(\pi/2))t^{2},$ $t\rightarrow +\infty$, the estimate \eqref{akotreba} implies $n\leq m.$ The positivity of polynomial $Q(\cdot)$ on the non-negative real axis implies $b_{m}>0$ so that \eqref{akotrebae} gives $a_{n}>0.$
Considering the asymptotic behaviour of terms appearing in \eqref{akotrebae}, we get that the inequality $n<m$ cannot be satisfied so that $m=n.$
 Dividing the both sides of \eqref{akotrebae} with $t^{2n}$ and letting $t\rightarrow +\infty$ in the obtained expression, we get that
$a_{n}/2b_{n} \geq +\infty,$ which is a contradiction. This completes the proof of (i). To prove (ii), suppose that the estimates
\begin{align*}
\ln (1+x) \geq \frac{P_{0}(x)}{Q_{0}(x)} \geq \frac{f(x)}{\sqrt{x+1}},\quad x\in (-1,0]
\end{align*}
hold for some real polynomials $P_{0}(\cdot)$ and $Q_{0}(\cdot)$ such that $Q_{0}(x) \neq 0$ for $x\in (-1,0].$ Then \eqref{mbswqq} holds for some real polynomials $P(\cdot)$ and $Q(\cdot)$ such that $Q(x) > 0$ for $x\in (-1,0].$  
Letting $x\rightarrow -1-$ in \eqref{mbswqq}, we get that $Q(-1)=0.$ If $P(x)=\sum_{j=0}^{n}a_{j}x^{j}$ and $Q(x)=\sum_{j=0}^{m}b_{j}x^{j}$ for some non-negative integers $m,\ n$ and some real numbers $a_{j}, \ b_{j}$ ($a_{n}b_{m}\neq 0;$ again, we cannot have $P(x)\equiv 0$), this implies
\begin{align}\label{mravi}
\ln (1+x) \cdot \sum_{j=0}^{m}b_{j}x^{j} \geq \sum_{j=0}^{n}a_{j}x^{j}\geq \frac{f(x)}{\sqrt{x+1}}\sum_{j=0}^{m}b_{j}x^{j},\quad x\in (-1,0].
\end{align}
Letting $x\rightarrow 0-$ in this expression, we get that $a_{0}=0$ so that $n\geq 1$ and $x | P(x).$ Define $P_{1}(x):=P(x)/x$ and $Q_{1}(x):=Q(x)/(x+1).$ Then $P_{1}(x)$ and $Q_{1}(x)$ are real polynomials, $Q_{1}(x)> 0$ for $x\in (-1,0]$ and after multiplication with $\frac{x+1}{xQ(x)}\leq 0$ the estimate \eqref{mravi} implies
\begin{align}\label{zeljo}
\frac{x+1}{x}\ln (1+x) \leq \frac{P_{1}(x)}{Q_{1}(x)}\leq \sqrt{x+1}\frac{f(x)}{x},\quad x\in (-1,0).
\end{align}
Letting $x\rightarrow -1-$ in this expression, we get that $\lim_{x\rightarrow -1-}\frac{P_{1}(x)}{Q_{1}(x)}=0,$ which implies 
$P_{1}(-1)=0.$ Since $P_{1}(x)$ is a non-zero polynomial, we get that $x+1 | P_{1}(x).$ Multiplying the equation \eqref{zeljo} with $\frac{x}{x+1}\leq 0,$
we get 
$$
\ln(1+x) \geq \frac{P_{1}(x)}{Q_{1}(x)}\geq \frac{f(x)}{\sqrt{x+1}},\quad x\in (-1,0).
$$
Letting $x\rightarrow 0-,$ we get 
$$
\ln(1+x) \geq \frac{P_{1}(x)}{Q_{1}(x)}\geq \frac{f(x)}{\sqrt{x+1}},\quad x\in (-1,0].
$$
Repeating this procedure, we get that for every natural number $k$ we have $(x+1)^{k} | Q(x),$ which is a contradiction.


\begin{thebibliography}{90}

\bibitem{2020}
C. Chesneau and Y. J. Bagul,
\emph{New sharp bounds for the logarithmic function,}
Electronic J. Math. Anal. Appl.
{\bf 8(1)} (2020), 140--145.

\bibitem{karamata}
J. Karamata,
\emph{Sur quelques probl\`emes pos\' es par Ramanujan,}
J. Indian Math. Soc. (N.S.)
{\bf 24} (1960), 343--365.

\bibitem{mitro}
D. S. Mitrinovi\' c,
\emph{Analytic Inequalities,} 
Springer, Berlin, 1970.

\bibitem{nova}
F. Topsoe,
\emph{Some bounds for the logarithmic function,} In: Cho YJ, Kim JK, Dragomir
SS Editors, Inequality theory and applications 4, New York: Nova Science Publishers, 137,
2007.

\end{thebibliography}
\end{document}